\documentclass[11pt,reqno,oneside]{amsart}

\usepackage[asymmetric,top=3.5cm,bottom=4.3cm,left=3.1cm,right=3.1cm]{geometry}
\usepackage{amsmath,amsfonts,amsthm,mathrsfs,amssymb,cite}
\usepackage[usenames]{color}

\usepackage{mathtools} %for \mathclap, \mathrlap, \smashoperator etc
\usepackage{graphics}
\usepackage{framed}

\usepackage{enumerate}

\usepackage{hyperref}
\usepackage{xcolor}
\hypersetup{
  colorlinks,
  linkcolor={blue!80!black},
  urlcolor={blue!80!black},
  citecolor={blue!80!black}
}
\usepackage[capitalize,noabbrev]{cleveref}

\usepackage{graphicx}
\usepackage{latexsym}

\newtheorem{definition}{Definition}[section]

\newtheorem{thm}{Theorem}[section]
\newtheorem{cor}[thm]{Corollary}

\newtheorem{Lemma}[thm]{Lemma}
\numberwithin{equation}{section}
\newtheorem{remark}[thm]{Remark}

\newtheorem{hypothesis}{Assumption}

\newcommand{\Bx}{\mathbf{x}}
\newcommand{\Bi}{\mathbf{i}}

\newcommand{\BA}{\mathbf{A}}

\newcommand{\BO}{\mathbf{0}}

\newcommand{\CR}{\mathcal{R}}

\newcommand{\RR}{\mathbb{R}}

\author{}
\date{Aug 2024}
\title{Determining a magnetic Schrödinger equation by a single far-field measurement}

\author{Chaohua Duan}
\address{Department of Mathematics, City University of Hong Kong, Kowloon, Hong Kong, China}
\email{chduan3@gmail.com, chduan3-c@my.cityu.edu.hk}

\author{Zhen Xue}
\address{1. School of Mathematics, North University of China, 2. Department of Mathematics, City University of Hong Kong, China}
\email{xuezhen@nuc.edu.cn}

\allowdisplaybreaks

\begin{document}

\begin{abstract}
  This paper investigates the inverse scattering problem for the magnetic Schrödinger equation. We first establish the well-posedness of the direct problem through a variational approach under physically meaningful assumptions on the magnetic and electric potentials. Our main results demonstrate that a single far-field measurement uniquely determines the support of the potential functions when the scatterer has polyhedral structures. 

  A significant theoretical byproduct of our analysis reveals that transmission eigenfunctions must vanish at corners in two dimensions and edge corners in three dimensions, provided the angle is not $\pi$. This geometric property of eigenfunctions extends previous results for the non-magnetic case and provides new insights into the interaction between quantum effects and singular geometries. The proof combines complex geometric optics solutions with careful asymptotic analysis near singular points.
  
  From an inverse problems perspective, our work shows that minimal measurement data suffices for shape reconstruction in important practical cases, advancing the theoretical understanding of inverse scattering with magnetic potentials. The results have potential applications in quantum imaging, material characterization, and nondestructive testing where magnetic fields play a crucial role.

        \medskip
        
    
    
\end{abstract}
    
 \maketitle 
 
\section{Introduction}
\subsection{Background}
Let $\Omega \subset \mathbb{R}^n$ ($n \geq 2$) be a bounded open set. We define the magnetic Schrödinger operator as follows:
\begin{equation}\label{MSO}
  \begin{aligned}
    L_{\BA, q}(\Bx, D) u(\Bx) & :=\sum_{j=1}^n\left(D_j+A_j(\Bx)\right)^2 u(\Bx)+q(\Bx) u(\Bx) \\
    & =-\Delta u(\Bx)+\BA(\Bx) \cdot D u(\Bx)+D \cdot(\BA(\Bx) u(\Bx))+\left((\BA(\Bx))^2+q(\Bx)\right) u(\Bx),
    \end{aligned}
\end{equation}
where $D = (D_j)_{j=1}^{n}$ with $D_j=\frac{1}{\Bi}\partial_{x_j}$ , $\BA =(A_j)_{j=1}^{n} \in L^{\infty}\left(\Omega, \mathbb{C}^n\right)$ is the magnetic potential, and $q \in L^{\infty}(\Omega, \mathbb{C})$ is the electric potential. 

Consider the magnetic Schrödinger equation:
\begin{equation}\label{main:eqs}
    \begin{cases}
      L_{\BA, q}(\Bx, D) u - k^2 u = 0\ &\ \mbox{in}\ \ \Omega,\medskip\\
      (\Delta+k^2) u=0\ &\ \mbox{in}\ \ \mathbb{R}^{n}\backslash \Omega,\medskip\\
      u^{+}=u^{-},\ \ \partial_\nu u^{+} - \Bi \nu \cdot (D+A)u^{-}=0\ &\ \mbox{on}\ \partial\Omega, \medskip\\
      u = u^{i} + u^{s} \ &\ \mbox{in}\ \ \mathbb{R}^{n}\backslash \Omega,
      \end{cases}
  \end{equation}
  where $u^{\pm}$ and $\partial_{\nu} u^{\pm}$ are the limits of $u$ and $\frac{\partial u}{\partial \nu}$ from the exterior $(+)$ and interior $(-)$ regions, respectively. The wavenumber $k \in \mathbb{R}_{+}$ characterizes the wave propagation frequency, and the scattered wave $u^s$ satisfies the Sommerfeld radiation condition:
  \begin{equation}\label{SRC}
    \lim_{r \to \infty} r^{(n-1)/2} (\frac{\partial u^{s}}{\partial r} - \Bi k u^{s}) = 0, \ r=|\Bx|,
  \end{equation}
  which holds uniformly with respect to $\hat{\Bx} = \Bx / |\Bx| \in \mathbf{S}^{n-1}$. Moreover, the following asymptotic expansion holds:
  \begin{equation}\label{FP}
    u^s(\Bx)=\frac{e^{i k|\Bx|}}{4 \pi |\Bx|^{(n-1)/2}} u^{\infty}(\hat{\Bx})+\mathcal{O}\left(\frac{1}{|\Bx|^{(n+1)/2}}\right),|\Bx| \rightarrow+\infty
  \end{equation}
  uniformly in all directions $\hat{\Bx} = \Bx / |\Bx|$. The function \( u^{\infty}(\hat{\mathbf{x}}) \) is known as the \textit{far-field pattern} or \text{scattering ampltitude} corresponding to $u^i$.
  
  In this paper, we focus on the inverse scattering problem of recovering the support of the magnetic potential $(\Omega, \BA, q)$ by knowledge of the far-field data $u_{\infty}(\hat{\Bx}; d, k)$; that is
\begin{equation}\label{ISP}
    u^{\infty}(\hat{\Bx}; u^i) \to (\Omega; \BA, q).
\end{equation}
Problem \eqref{ISP} is termed a single far-field measurement when the far-field pattern corresponds to a single incident wave $u^i$. In contrast, it is referred to as many far-field measurements when multiple incident waves are involved. It is worth noting that problem \eqref{ISP} is nonlinear and well-known for its ill-posedness \cite{DCRK, IV}. The determination of shape from minimal or optimal measurement data remains an unresolved challenge in inverse scattering theory \cite{DCRK, IV}. Notably, significant progress has been made in recent years for the general transport eigenvalue problem (excluding magnetic and electric potentials), particularly in the recovery of impenetrable polyhedral scatterers using a minimized set of far-field measurements \cite{AGLR, JCMY, LHJZ, LHJZ1}. This research area is of considerable importance in various scientific and technological domains, including radar and sonar, geophysical exploration, medical imaging, nondestructive testing, and remote sensing \cite{AHHK, AHHK1, DCRK, IV}, as outlined in the references cited.

We would like to emphasize that the inverse problem under consideration faces fundamental theoretical limitations, even with complete far-field data from infinitely many incident plane waves. As demonstrated in the seminal work by \cite{Sun1993}, the non-uniqueness issue persists for magnetic Schrödinger operators due to the gauge invariance of vector potentials. This obstruction is intrinsic to problems involving magnetic potentials, where different configurations of $\BA$ and $q$ can produce identical far-field measurements \cite{Uhlmann2003,Krupchyk2012}. Recent advances in \cite{Cakoni2016} have further characterized these limitations for polygonal geometries similar to our setting. These theoretical barriers motivate our focus on shape reconstruction rather than full potential recovery in this work.

\subsection{Statement of the main results and discussions}

In this paper, we consider the scenario of minimal measurement data, specifically a single far-field pattern. Here, $u^{\infty}(\hat{\Bx}; u^i)$ is known for all $\hat{\Bx} \in \mathbb{S}^{n-1}$ and a single fixed incident wave $u^i$. We aim to establish the uniqueness of the scattering area under certain admissible conditions. During the proof of uniqueness, we discovered interesting byproducts revealing geometric structures of the eigenfunctions. Specifically, the corresponding transmission eigenfunction must vanish near corners if the angle is not $\pi$ (i.e., not a straight angle). This study provides a novel investigation into the nullity phenomena of eigenfunctions associated with magnetic Schrödinger operators, diverging from the conventional focus on the geometric attributes of transmission eigenfunctions as documented in \cite{BE, DCL, DDL, Liu, BL1, BLLW}. Our findings have significant implications for condensed matter physics, quantum mechanics, and mathematical physics. From a practical perspective, they enhance the understanding of quantum transport phenomena, particle behavior, and the properties of solutions to the magnetic Schrödinger equation. For example, our results may explain the disappearance of electronic states under strong magnetic fields, which could impact material conductivity and magnetic properties.

The magnetic Schrödinger equation extends the Schrödinger equation by incorporating classical electromagnetic fields into quantum mechanics. It is widely used to study the quantum behavior of charged particles in magnetic fields, enabling the determination of wave functions and the prediction of physical quantities such as position, momentum, and energy. References \cite{IVV}, \cite{LASJ}, and \cite{AWGEL} discuss the properties of the magnetic Schrödinger operator under strong magnetic fields, semiclassical estimates in one dimension, and applications in superconducting materials, respectively. 

The theorem presented herein explores the implications of reconstructing the shape of a scattering area in two and three dimensions using a single far-field measurement.

\begin{thm}\label{main:thm2}
  Consider the scattering problem described by equation \eqref{main:eqs} associated with two scatterers $\left(\Omega_j ; k, d, q_j, \BA_j\right)$ for $ j=1,2$, in $\mathbb{R}^n$ with $n=2,3$.
  If the far-field patterns corresponding to the two scatterers are the same, i.e. $u_{\infty}^1\left(\hat{\Bx} ; u^i\right)=u_{\infty}^2\left(\hat{\Bx} ; u^i\right)$, then $\Omega_1 \Delta \Omega_2:=\left(\Omega_1 \backslash \Omega_2\right) \cup\left(\Omega_2 \backslash \Omega_1\right)$ cannot possess a corner in two dimensions or an edge corner in three dimensions.
\end{thm}

More detail result is given in Theorem \ref{main:thm4}. 

\begin{remark}
  From Theorem \ref{main:thm2}, we can further get that $\Omega_1=\Omega_2$ if $\Omega_1$ and $\Omega_2$ are convex polygons in $\mathbb{R}^2$ or rectangular boxs in $\mathbb{R}^3$.
\end{remark}

The following theorem is an interesting byproduct of proving the uniqueness of the inverse scattering problem. It reveals a geometric property of the eigenfunctions, specifically that they vanish near corners.

\begin{thm}\label{main:thm3}
  Consider the magnetic Schrödinger equation given by \eqref{main:eqs} with $n=2,3$. Let $\Bx_c \in \Omega$ be a corner point, and let $\mathcal{N}_h$ be a neighbourhood of $\Bx_c$ within $\Omega$, which is the area where a circle or sphere with radius h intersects $\Omega$ with $h\in \RR_+$ is sufficiently small. Suppose that $v\in H^2(\mathcal{N}_h)$, $w\in H^1(\mathcal{N}_h)$, $\nabla w \in {L^{{\frac{2}{1-\varepsilon}}}(\mathcal{N}_h)}$ for $0<\varepsilon <1$, and $\mathbf{A}\in H^{2}(\mathcal{N}_h)$, then it follows that $v(\Bx_c)=w(\Bx_c) = 0$.
\end{thm}

More detailed results are respectively given in Theorems \ref{thm:2D} and \ref{thm:3D} for the two and
three dimensions.

\begin{remark}
  If $\BA \equiv \mathbf{0}$, then the equation  \eqref{main:eqs} reduces to an interior transmission eigenvalue problem. The vanishing property of the associated eigenfunctions has been established; see \cite[Corollary 2.4,Corollary 3.2]{DDL} for details.
\end{remark}

This paper is structured as follows. Section 2 addresses the well-posedness of the direct problem, covering the existence and uniqueness of solutions, as well as their norm estimation. Section 3 explores the geometric properties of eigenfunctions, with a particular emphasis on their vanishing behavior at corner points. Section 4 presents a discussion on the unique recovery results for the inverse scattering problem.

\section{well-posedness of the forward problem}
This section provides proofs for the existence and uniqueness of solutions to the direct scattering problem, along with the presentation of norm control for these solutions. 

Since the incident wave $u^i$ satisfies the Helmholtz equation $\Delta u^i+k^2 u^i=0$ in $\mathbb{R}^n$, the scattered wave $u^s$ satisfies the following problem with the boundary data $f_2=\Bi \nu \cdot \BA u^i$ on $\partial \Omega$ and the source term $f_1=\Bi \nabla \cdot (\BA u^{i}) + \Bi \BA \cdot \nabla u^{i} + (-\BA^2 -q)u^{i}$ in $\Omega$ :
\begin{align}
  \Delta u ^{s} + k^2 u^{s} &= 0 \ \ \mbox{in}\ \ \mathbb{R}^{n} \backslash \overline{\Omega},  \label{eq1}  \\[10pt]
  \Delta u^{s} + \Bi \nabla \cdot (\BA u^{s}) + \Bi \BA \cdot \nabla u^{s} + (k^2 -\BA^2 -q)u^{s} &= -f_1 \ \mbox{in}\ \ \Omega,\\[10pt]   
  u_{+}^{s} - u_{-}^{s} =0, \ \frac{\partial u_{+}^{s}}{\partial \nu} -  \frac{\partial u_{-}^{s}}{\partial \nu} - \Bi (\nu \cdot \BA) u_{-}^{s} &=f_2 \ \ \mbox{on} \ \partial \Omega. \label{BCond1}
\end{align}

For $f_1\in L^{2}(\Omega)$ and $f_{2} \in H^{-1/2}(\partial \Omega)$, the solution $u^{s} \in H^{1}_{loc}(\mathbb{R}^{n})$ of \eqref{eq1}-\eqref{BCond1} satisfies the following variational form
\begin{equation}\label{wsol}
  \begin{split}
  & \iint_\Omega\left[\nabla u^s \cdot \nabla \bar{\varphi} + \Bi \BA \cdot (u^s \nabla \bar{\varphi} - \nabla u^s \bar{\varphi}) + (\BA^2 +q - k^2)u^s \bar{\varphi} \right] \mathrm{d} x \\
  &+ \iint_{\mathbb{R}^n \backslash \bar{\Omega}}\left[\nabla u^s \cdot\right.  \left.\nabla \bar{\varphi}-k^2 u^s \bar{\varphi}\right] \mathrm{d} x =   \iint_\Omega f_1 \bar{\varphi} \mathrm{d} x - \int_{\partial \Omega} f_2 \bar{\varphi} \mathrm{d} s
  \end{split}
  \end{equation}
for any test function $\varphi \in H^{1}(\mathbb{R}^{n})$ with compact support. A well known regularity result for elliptic differential equations \cite{MW} yields that $u^s$ is even analytic in $\mathbb{R}^n \setminus \Omega$. The following assumption gives the well-posed condition for the direct problem.

\begin{hypothesis}\label{WPD}
  The electric potential and magnetic potential satisfy the following conditions:
  \begin{itemize}
    \item The electric potential $q \in L^{\infty}(\Omega)$ satisfies $\Im (q) \leq 0$ almost everywhere in $\Omega$.
    \item The magnetic potential $\BA \in L^{\infty}(\Omega) \cap H^1(\Omega)$ satisfies $\Im (\BA) \leq \mathbf{0}$ and $\Im (\BA^2)  \leq \mathbf{0}$ almost everywhere in $\Omega$.
  \end{itemize}
\end{hypothesis}

\begin{thm}
  For any $f_1 \in L^2(\Omega)$ and $f_2 \in H^{-1 / 2}(\partial \Omega)$, if the electric potential $q$ and the magnetic potential $\BA$ satisfy the Assumption \ref{WPD}, then there exists at most one solution $v \in H_{loc}^1\left(\mathbb{R}^n\right)$ to the problem described by \eqref{eq1}-\eqref{BCond1}, or, equivalently the problem described by \eqref{wsol} and \eqref{SRC} has a unique solution in $H_{loc}^1\left(\mathbb{R}^n\right)$.
\end{thm}

\begin{proof}
  Let $v_1$ and $v_2$ are two solutions to the problem described by \eqref{eq1}-\eqref{BCond1}. Define $v=v_1 - v_2$, then $v$ satisfies \eqref{wsol}, \eqref{SRC} with $f_2=0$ on $\partial \Omega$ and $f_1=0$ in $\Omega$. To prove the uniqueness, we show that $v$ must vanish in all of $\mathbb{R}^n$.

Choose a ball $B_R$ centered at the origin such that $\bar{\Omega} \subset B_R$. Let $\phi \in C^{\infty}\left(\mathbb{R}^n\right)$ be a smooth function such that $\phi(\Bx)=1$ for $|\Bx|<R$ and $\phi(\Bx)=0$ for $|\Bx| \geqslant R+1$. Define $\varphi=\phi v$. Then, substituting $\varphi=\phi v$  into \eqref{wsol}, we obtain
\begin{equation}
  \begin{split}
&\iint_{R \leqslant |\Bx| \leqslant R+1}\left[\nabla v \cdot \nabla \bar{\varphi}-k^2 v \bar{\varphi}\right] \mathrm{d} x+\iint_{B_R \backslash \bar{\Omega}}\left[|\nabla v|^2-k^2|v|^2\right] \mathrm{d} x \\
&+ \iint_\Omega\left[|\nabla v|^2 +\Bi \BA (v\nabla \bar{v} - \nabla  v \bar{v}) +(\BA^2 +q -k^2)|v|^2  \right] \mathrm{d} x =0 .
\end{split}
\end{equation}

From the interior regularity results in \eqref{SRC}, $v$ is analytic outside of $B_R$. Applying Green's first theorem to the first integral yields
\begin{equation}
\begin{split}
  &-\int_{|\Bx|=R} \frac{\partial v}{\partial \nu} \bar{v} \mathrm{~d} s +\iint_{B_R \backslash \bar{\Omega}}\left[|\nabla v|^2-k^2|v|^2\right] \mathrm{d} x \\
  &+ \iint_\Omega\left[|\nabla v|^2 +\Bi \BA (v\nabla \bar{v} - \nabla  v \bar{v}) +(\BA^2 +q -k^2)|v|^2  \right] \mathrm{d} x =0.
\end{split}
\end{equation}
By the Assumption \ref{WPD} and the Cauchy's inequality, it follows that 
$$
\Im \int_{\partial B_R} v \frac{\partial \bar{v}}{\partial \nu} \mathrm{d} s \geqslant 0 .
$$

It follows that  $v=0$ in $\mathbb{R}^n \backslash B_R$ by \cite[Theorem 2.13]{DCRK}. By analytic continuation $v=0$ in $\mathbb{R}^n \backslash \Omega$, and thus the trace of $v$ also vanishes on $\partial \Omega$. Therefore, $v \in H^1\left(\mathbb{R}^n\right)$ is a weak solution of
 $$\Delta u + \Bi \nabla \cdot (\BA u) + \Bi \BA \cdot u + (k^2 -\BA^2 -q)u=0$$
 in $\mathbb{R}^n$ and is identically zero outside some ball. The unique continuation principle implies that $v$ vanishes in all of $\mathbb{R}^n$. This completes the proof.
\end{proof}

Following this, we introduce the Dirichlet-to-Neumann operator (DtN) to prove the existence of solutions for \eqref{eq1}-\eqref{BCond1} and \eqref{SRC}.  It is denoted as follows:
\begin{equation*}
  \Lambda : v \mapsto \frac{\partial \tilde{v}}{\partial \nu} \text { on } \partial B_R
\end{equation*} 
mapping $v$ to $\partial \tilde{v} / \partial \nu$, where $\tilde{v}$ solves the exterior Dirichlet problem for the Helmholtz equation $\Delta \tilde{v}+k^2 \tilde{v}=0$ in $\mathbb{R}^n \backslash B_R$ with the Dirichlet boundary data $\left.\widetilde{v}\right|_{\partial B_R}=v$. Again here, $B_R$ is a ball of radius $R$, such that $\overline{{\Omega}} \subset B_R$. The following lemma gives an important property of the DtN operator, see \cite{BOLX} or \cite[Theorem 3.13]{DCRK} for details.
\begin{Lemma}\label{DtN}
  The DtN operator $\Lambda$ is a bijective bounded linear operator from $H^{1 / 2}\left(\partial B_R\right)$ to $H^{-1 / 2}\left(\partial B_R\right)$. Furthermore, there exists a bounded operator $\Lambda_0: H^{1 / 2}\left(\partial B_R\right) \rightarrow H^{-1 / 2}\left(\partial B_R\right)$ satisfying that
$$
-\int_{\partial B_R} \Lambda_0 w \bar{w} \mathrm{~d} s \geqslant c\|w\|_{H^{1 / 2}\left(\partial B_R\right)}^2
$$
for some constant $c>0$, such that $\Lambda-\Lambda_0: H^{1 / 2}\left(\partial B_R\right) \rightarrow H^{-1 / 2}\left(\partial B_R\right)$ is compact.
\end{Lemma}

We reformulate problem \eqref{eq1}-\eqref{BCond1} as follows: given $f_1 \in L^2(\Omega)$ and $f_2 \in$ $H^{-1 / 2}(\partial \Omega)$, find $v \in H^1\left(B_R\right)$ satisfying \eqref{eq1}-\eqref{BCond1} and the boundary condition 
\begin{equation}\label{BCond2}
  \frac{\partial v}{\partial \nu}=\Lambda v \quad \text { on } \partial B_R .
\end{equation}
Again, problem \eqref{eq1}-\eqref{BCond1}, \eqref{BCond2} satisfies the following variational form: $v\in H^{1}(B_R)$ and
\begin{equation}\label{VWS}
    \begin{split}
      &\iint_{B_{R} \backslash \bar{\Omega}}\left[\nabla v \cdot\right.  \left.\nabla \bar{\varphi}-k^2 v \bar{\varphi}\right] \mathrm{d} x - \int_{\partial B_R} \Lambda v \bar{\varphi} \mathrm{d} s\\
      &+   \iint_\Omega\left[\nabla v \cdot \nabla \bar{\varphi} + \Bi \BA (v \nabla \bar{\varphi} - \nabla v \bar{\varphi}) + (\BA^2 +q - k^2)v \bar{\varphi} \right] \mathrm{d} x \\
      & =   \iint_\Omega f_1 \bar{\varphi} \mathrm{d} x - \int_{\partial \Omega} f_2 \bar{\varphi} \mathrm{d} s
      \end{split}
\end{equation}
By \cite[Lemma 5.22]{CFCD}, a solution \( v \) to problem \eqref{eq1}-\eqref{BCond1} and \eqref{BCond2} can be extended to a solution of the scattering problem \eqref{eq1}-\eqref{BCond1} and \eqref{SRC}. Conversely, a solution \( v \) to the scattering problem \eqref{eq1}-\eqref{BCond1} and \eqref{SRC}, restricted to \( B_R \), solves problem \eqref{eq1}-\eqref{BCond1} and \eqref{BCond2}. 

For existence, we present the following theorem:
\begin{thm}\label{SolE}
Let $f_1 \in L^2(\Omega)$, $f_2 \in H^{-1 / 2}(\partial \Omega)$, the electric potential $q$ and the magnetic potential $\BA$ satisfy the Assumption \ref{WPD}. Then problem \eqref{VWS} has a unique solution $v \in H^1\left(B_R\right)$. Furthermore,
\begin{equation}\label{Es}
  \|v\|_{H^1\left(B_R\right)} \leqslant C\left(\left\|f_2\right\|_{H^{-1 / 2}(\partial \Omega)}+\|f_1\|_{L^2(\Omega)}\right)
\end{equation}
with a positive constant $C$ independent of $f_1$ and $f_2$.
\end{thm}

\begin{proof}
  We write \eqref{VWS} as 
  \begin{equation}\label{VE}
    a(v,\varphi) = b(\varphi)
  \end{equation}
with 
\begin{equation*}
  \begin{split}
    a(v,\varphi) = &\iint_{B_{R} \backslash \bar{\Omega}}\left[\nabla v \cdot\right.  \left.\nabla \bar{\varphi}-k^2 v \bar{\varphi}\right] \mathrm{d} x - \int_{\partial B_R} \Lambda v \bar{\varphi} \mathrm{d} s\\
      &+   \iint_\Omega\left[\nabla v \cdot \nabla \bar{\varphi} + \Bi \BA (v \nabla \bar{\varphi} - \nabla v \bar{\varphi}) + (\BA^2 +q - k^2)v \bar{\varphi} \right] \mathrm{d} x
  \end{split}
\end{equation*}
and
\begin{equation*}
  b(\varphi) = \iint_\Omega f_1 \bar{\varphi} \mathrm{d} x - \int_{\partial \Omega} f_2 \bar{\varphi} \mathrm{d} s.
\end{equation*}
Denote $a_1 + a_2 = a$, where
\begin{equation*}
  \begin{split}
    a_1(v,\varphi) =& \iint_{B_{R} \backslash \bar{\Omega}}\left[\nabla v \cdot\right.  \left.\nabla \bar{\varphi}+ v \bar{\varphi}\right] \mathrm{d} x - \int_{\partial B_R} \Lambda_0 v \bar{\varphi} \mathrm{d} s\\
    &+   \iint_\Omega\left[\nabla v \cdot \nabla \bar{\varphi} + \Bi \BA (v \nabla \bar{\varphi} - \nabla v \bar{\varphi}) + (2|\BA|^{2} +\frac{1}{2})v \bar{\varphi} \right] \mathrm{d} x,
  \end{split}
\end{equation*}
and 
\begin{equation*}
  \begin{split}
    a_2(v,\varphi) = & - \iint_{B_{R} \backslash \bar{\Omega}} [1+ k^2 ] v \bar{\varphi} \mathrm{d} x - \int_{\partial B_{R}} (\Lambda - \Lambda_{0}) v \bar{\varphi} \mathrm{d}s \\
    & + \iint_\Omega (\BA^2 -2|\BA|^{2} + q -k^2 -\frac{1}{2}) v \bar{\varphi} \mathrm{d} x,
  \end{split}
\end{equation*}
where $\Lambda_0$ is the operator defined in Lemma \ref{DtN}. By the boundedness of $\Lambda_0$ and trace theorem, $a_1$ is bounded.  By the Riesz representation theorem, there exists a bounded linear operator $A_1: H^1\left(B_R\right) \rightarrow H^1\left(B_R\right)$, such that
$$
a_1(v, \varphi)=\left(A_1 v, \varphi\right) \quad \text { for all } \varphi \in H^1\left(B_R\right) .
$$
Using the Cauchy-Schwartz inequality, on can drive that
\begin{equation*}
  |\iint_\Omega \Bi \BA (v \nabla \bar{\varphi} - \nabla v \bar{\varphi}) \mathrm{d} x | \leq \iint_\Omega 2 |\BA|^{2} |v|^2 + \frac{1}{2}|\nabla v|^2 \mathrm{d} x .
\end{equation*}

By the assumption, we can further get that
\begin{equation*}
  \begin{split}
    \Re{[a_1(v,v)]} \geq  \|v\|^{2}_{H^1(B_R\backslash \bar{\Omega})} - \int_{\partial B_{R}} \Lambda_{0} v \bar{\varphi} \mathrm{d}s + \frac{1}{2} \|v\|^{2}_{H^1(\Omega)} \geq  \frac{1}{2} \|v\|^{2}_{H^1(B_{R})},
  \end{split}
\end{equation*}
that is, $a_1$ is strictly coercive. The Lax-Milgram theorem (see  in \cite[Theorem 13.26]{KMK}) implies that the operator $A_1: H^1\left(B_R\right) \rightarrow H^1\left(B_R\right)$ has a bounded inverse. By the Riesz representation theorem again, there exists a bounded linear operator $A_2: H^1\left(B_R\right) \rightarrow H^1\left(B_R\right)$, such that
$$
a_2(v, \varphi)=\left(A_2 v, \varphi\right) \quad \text { for all } \varphi \in H^1\left(B_R\right) .
$$

By the compactness of $\Lambda-\Lambda_0$ and Rellich's embedding theorem (that is, the embedding of $H^1\left(B_R\right)$ into $L^2\left(B_R\right)$) is compact, it follows that $A_2$ is compact. By the Riesz representation theorem again, there is a function $\tilde{v} \in H^1\left(B_R\right)$, such that
$$
b(\varphi)=(\widetilde{v}, \varphi) \quad \text { for all } \varphi \in H^1\left(B_R\right) .
$$

Thus, the variational formulation \eqref{VE} is equivalent to the problem
\begin{equation}\label{VE2}
  \text { find } v \in H^1\left(B_R\right) \text { such that } A_1 v+A_2 v=\widetilde{v} \text {, }
\end{equation}
where $A_1$ is bounded and strictly coercive and $A_2$ is compact. The Riesz-Fredholm theory and the uniqueness result imply that problem \eqref{VE2} or, equivalently, there is a unique solution for problem \eqref{VE}. Estimate \eqref{Es} follows from the fact that $\|\widetilde{v}\|_{H^1\left(B_R\right)}=\|b\|$ is bounded by $\|f_1\|_{L^2(D)}+\left\|f_2\right\|_{H^{-1 / 2}(\partial D)}$. 

\end{proof}

\section{Vanishing Properties}
In this section, we present a detailed description and proof of Theorem \ref{main:thm3}, which establishes the vanishing properties of eigenfunctions at corner points in both two and three dimensions. These results reveal that, under specific conditions, the eigenfunctions must vanish at corners, offering new insights into the geometric structure of solutions to the magnetic Schrödinger equation.

We first establish a set of geometric notations to be used throughout the paper. Let $(r, \theta)$ denote the polar coordinates in $\mathbb{R}^2$, so that $\mathbf{x}=r(\cos \theta, \sin \theta)$. For $\mathbf{x} \in \mathbb{R}^2, B_h(\mathbf{x})$ represents the open disk of radius $h \in \mathbb{R}_{+}$ centered at $\mathbf{x}$. Consider an open sector in $\mathbb{R}^2$ defined as follows:
\begin{equation}\label{W}
  W = \big\{ \Bx \in \mathbb{R}^{2} \;\big|\; 
  \Bx \neq \mathbf{0}, \ 
  \theta_{m} < \arg(x_1 + \mathbf{i} x_2) < \theta_{M} \big\},
\end{equation}
where $-\pi \leq \theta_{m}<\theta_{M}<\pi, \mathbf{i}:=\sqrt{-1}$ and let $\Gamma^{+}$ and $\Gamma^{-}$ respectively denote the curves $\left(r, \theta_{M}\right)$ and $\left(r, \theta_{m}\right)$ with $r>0 .$ Define 
\begin{equation}\label{Sh}
    S_{h}=W \cap B_{h}, \Gamma_{h}^{\pm}=\Gamma^{\pm} \cap B_{h}, \bar{S}_{h}=\overline{W} \cap B_{h}, \Lambda_{h}=S_{h} \cap \partial B_{h}, \text { and } \Sigma_{\Lambda_{h}}=S_{h} \backslash S_{h / 2}.
\end{equation}
where $B_h:=B_h(\mathbf{0})$. In the subsequent discussion, the sector $S_h$ denotes a neighborhood surrounding a vertex corner of the domain $\Omega$. For the remainder of this paper, we make the assumption that $h\in \RR_+$ is sufficiently small, ensuring that $S_h$ remains entirely within the domain $\Omega$. As a result, $\Gamma_h^{\pm}$ will lie exclusively on the two edges associated with the aforementioned vertex, which should be evident from the context. Please refer to Figure \ref{fg:01} for a 2D illustration depicting the geometry.

\begin{figure}[htbp]
  \centering
  \includegraphics[scale=0.25]{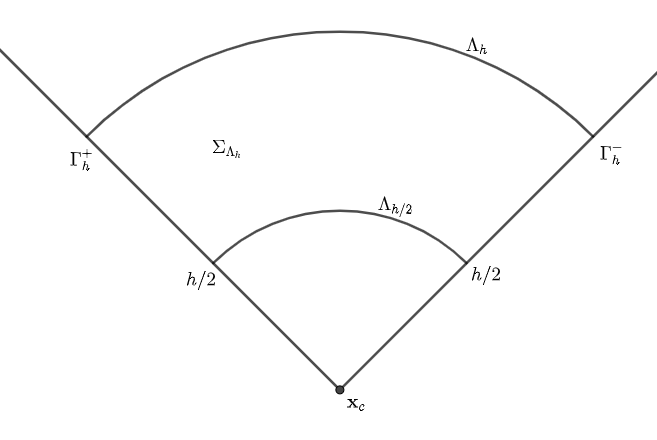}
  \caption{Illustration of the geometry in 2D} \label{fg:01}
\end{figure}

\subsection{Vanishing properties in 2D}

To establish our main results, we will utilize a complex geometrical optics (CGO) solution, the logarithm of which is a branch of the square root. This solution, introduced in \cite{ref2}, will be employed in conjunction with the lemma presented below.
\newtheorem{lemma}{Lemma}[section]
\begin{Lemma}\label{Lemma:U0}
    \cite[Lemma 2.2]{ref2} For $\Bx\in \RR^2$ denote $r=|\Bx|,\theta=arg(x_1+\mathbf{i}x_2).$ Define
\begin{equation}\label{Def:U0}
    u_{0}(\Bx):=\exp \left(\sqrt{r}\left(\cos \left(\frac{\theta}{2}+\pi\right)+\mathbf{i} \sin \left(\frac{\theta}{2}+\pi\right)\right)\right),
\end{equation}
then $\Delta u_0=0 \text { in } \RR^2 \setminus \RR_{0,-}^{2}$, where $\RR_{0,-}^{2}:=\left\{\Bx \in \RR^2|\ \Bx = (x_1,x_2);\ x_1{\leq}0, x_2=0\right\}$ and $s \mapsto u_{0}(sx)$ decays exponentially in $\RR_{+} .$ Choose $\alpha, s>0$, then there holds
\begin{equation}\label{norm:U0_1}
    \int_{W}\left|u_{0}(s \Bx) \| \Bx\right|^{\alpha} \mathrm{d} \Bx \leq \frac{2\left(\theta_{M}-\theta_{m}\right) \Gamma(2 \alpha+4)}{\delta_{W}^{2 \alpha+4}} s^{-\alpha-2}.
\end{equation}
\text {where } we define $\delta_{W}=:-\max _{\theta_{m}<\theta<\theta_{M}} \cos (\theta / 2+\pi)>0 . $ Moreover, one has
\begin{equation}\label{norm:U0_2}
    \int_{W} u_{0}(s \Bx) \mathrm{d} x=6 \mathbf{i}\left(e^{-2 \theta_{M} \mathbf{i}}-e^{-2 \theta_{m} \mathbf{i}}\right) s^{-2},
\end{equation}
and for a positive constant $h>0$, there holds
\begin{equation}\label{norm:U0_3}
    \int_{W \backslash B_{h}}\left|u_{0}(s \Bx)\right| \mathrm{d} \Bx \leq \frac{6\left(\theta_{M}-\theta_{m}\right)}{\delta_{W}^{4}} s^{-2} e^{-\delta_{W} \sqrt{h s} / 2}.
\end{equation}
\end{Lemma}

Furthermore, we have the following estimates: the first formula can also be found in \cite[Lemma 2.3]{DCL}, while the second through fourth formulas are detailed in \cite[Corollary 2.2]{DDL}.
\begin{cor}\label{cor2.2}
    The following estimates hold for the norm of $u_0$,
    \begin{align}
      \left\||\Bx|^{\alpha} u_{0}(s \Bx)\right\|_{L^{2}\left(S_{h}\right)}^{2} & \leq s^{-(2 \alpha+2)} \frac{2\left(\theta_{M}-\theta_{m}\right)}{\left(4 \delta_{W}^{2}\right)^{2 \alpha+2}} \Gamma(4 \alpha+4) \nonumber \\
       \left\|u_{0}(s \Bx)\right\|_{L^{2}\left(\Lambda_{h}\right)} & \leq \sqrt{h} e^{-\delta_{W} \sqrt{s h}} \sqrt{\theta_{M}-\theta_{m}},\nonumber \\
       \left\|\partial_{\nu} u_{0}(s \Bx)\right\|_{L^{2}\left(\Lambda_{h}\right)} & \leq \frac{1}{2} \sqrt{s} e^{-\delta_{W} \sqrt{s h}} \sqrt{\theta_{M}-\theta_{m}},\nonumber \\
        \left\|\partial_{\theta} u_{0}(s \Bx)\right\|_{L^{2}\left(\Lambda_{h}\right)} & \leq  \frac{\sqrt{s}}{2}h^2 e^{-\delta_{W} \sqrt{s h}} \sqrt{\theta_{M}-\theta_{m}} \nonumber \\
        \|u_{0}(s\Bx)\|_{L^{{\frac{2}{1+\varepsilon }}}(S_h)} &\leq \mathcal{O}(s^{-1-\varepsilon}),\quad \text{as} \quad s\to \infty, \nonumber \\
        \|u_{0}(s\Bx)\|_{H^{1}(\Gamma_{h}^{\pm})} &\leq C ,\quad \text{as} \quad s\to \infty, \nonumber \\
    \end{align}
 where $0 < \varepsilon < 1$, $\delta_{W}$ is defined in (\ref{norm:U0_1}) and $S_{h}$, $\Lambda_{h}$ are defined in (\ref{Sh}).
\end{cor}

\begin{proof}
  
  By direct calculation, we can deduce
  \begin{equation*}
    \begin{split}
        \|u_{0}(s\Bx)\|_{L^{{\frac{2}{1+\varepsilon }}}(S_h)} &=[ \int_{0}^{h} r \mathrm{d} r \int_{\theta_{m}}^{\theta_{M}} e^{{\frac{2}{1+\varepsilon }} \sqrt{s r} \cos (\theta / 2+\pi)} \mathrm{d} \theta ]^{{\frac{1+\varepsilon }{2}}}\\
        &\leq [\int_{0}^{h} r \mathrm{d} r \int_{\theta_{m}}^{\theta_{M}} e^{-{\frac{2}{1+\varepsilon }} \sqrt{s r} \delta_{W}} \mathrm{d} \theta]^{{\frac{1+\varepsilon }{2}}}\\
        &\leq [(\theta_{M}-\theta_{m}) \int_{0}^{\sqrt{sh}} \frac{2}{s^2}  e^{-{\frac{2}{1+\varepsilon }}t\delta_{W}} t^{3}  dt    ]^{{\frac{1+\varepsilon }{2}}}\\
        &\leq C s^{-1-\varepsilon},
    \end{split}
  \end{equation*}
  where $C>0$. Next, on $\Gamma_{h}^{+}$, it can be seen that
  \begin{equation*}
    \begin{split}
        |\partial_{\theta}u_{0}(s\Bx)| &=|\left|\frac{\sqrt{sr} }{2} e^{\sqrt{s r} \exp (i(\theta_{M} / 2+\pi))}\right| \leq \frac{\sqrt{sr} }{2 } e^{-\delta_{W} \sqrt{s r}},\\
        |u_{0}(s\Bx)|&= e^{\sqrt{sr}\cos(\theta_{M}/2+\pi)} \leq e^{-\delta_{W}\sqrt{sr}}.
    \end{split}
  \end{equation*}
  On $\Gamma_{h}^{+}$, the exterior unit normal $\nu := \nu^{+} = (-\sin{\theta_M},\cos{\theta_M})$, therefore, one can get 
  \begin{equation*}
    |\partial_{\nu}u_{0}(s\Bx)|=|\nu^{+} \cdot \nabla u_{0}(s\Bx) | = |\frac{\partial_{\theta}u_{0}(s\Bx)}{r}| \leq \frac{\sqrt{sr} }{2r } e^{-\delta_{W} \sqrt{s r}}.
  \end{equation*}
  Similarly, on $\Gamma_{h}^{-}$ we can deduce
  \begin{equation*}
    |\partial_{\nu}u_{0}(s\Bx)| \leq \frac{\sqrt{sr} }{2r } e^{-\delta_{W} \sqrt{s r}}.
  \end{equation*}
  Therefore, one can get 
  \begin{equation*}
        \|\partial_{\theta}u_{0}(s\Bx)\|_{L^{2}(\Gamma_{h}^{\pm})}^{2} \leq  \int_{0}^{h}  \frac{sr}{4} e^{-2\delta_{W} \sqrt{s r}} \mathrm{d} r 
        =  \int_{0}^{\sqrt{sh}} \frac{1}{2s} t^{3} e^{-2\delta_{W} t} \mathrm{d} t
        \leq  \frac{3}{4 \delta_{W}^{3} s } ,
  \end{equation*}
  and 
  \begin{equation*}
        \|u_{0}(s\Bx)\|_{L^{2}(\Gamma_{h}^{\pm})}^{2} \leq \int_{0}^{h}  e^{-2\delta_{W} \sqrt{s r}} \mathrm{d} r
        =  \int_{0}^{\sqrt{sh}} \frac{2}{s} t e^{-2\delta_{W} \sqrt{s r}} \mathrm{d} r 
        \leq  \frac{1}{ \delta_{w}^{2}s } .
  \end{equation*}
  Similarly, to calculate the $L^{2}$ norm of $\partial_{\nu}u_{0}(s\Bx)$ on $\Gamma_{h}^{\pm}$, we proceed as follows:
  \begin{equation*}
        \|\partial_{\nu}u_{0}(s\Bx)\|_{L^{2}(\Gamma_{h}^{\pm})}^{2} \leq  \int_{0}^{h} r \frac{sr}{4r^{2}} e^{-2\delta_{W} \sqrt{s r}} \mathrm{d} r 
        =  \int_{0}^{\sqrt{sh}} \frac{1}{2} t e^{-2\delta_{W} t} \mathrm{d} t
        \leq  \frac{1}{4\delta_{W}^{2}}.
  \end{equation*}
  Therefore, on $\Gamma_{h}^{\pm}$, when $s\to \infty$, the $H^{1}$ norm of $u_{0}(s\Bx)$ is bounded.
  The proof is complete. 
  
  \end{proof}

The following lemma shows the elementary expansion result for functions with $C^\alpha$ regularity: 
  \begin{Lemma}\label{Lemma:expansion}
     \cite[Lemma 2.4]{DDL} Suppose $f \in C^{\alpha}$, then the following expansion holds near the origin:
      \begin{equation}
              f(\Bx)=f(\mathbf{0})+\delta f(\Bx),\quad |\delta f(\Bx)|\leq \left\|f\right\|_{C^\alpha} |\Bx|^{\alpha}.
      \end{equation}
  \end{Lemma}

With the preliminary results established, we now present our first main result concerning the vanishing properties in two dimensions. This result concretely manifests as Theorem \ref{main:thm3} in two dimensions.

\begin{thm}\label{thm:2D}
  Let $S_{h}$ be defined in \eqref{Sh}. Suppose that $v\in H^2(S_h)$ and $w\in H^1(S_h) $ satisfy the following PDE system:
  \begin{equation}\label{main:TEF}
    \begin{cases}
      L_{\BA, q}(\Bx, D) w - k^2 w = 0\ &\ \mbox{in}\ \ S_h,\medskip\\
      (\Delta+k^2) v=0\ &\ \mbox{in}\ \ S_h,\medskip\\
      v=w,\ \ \partial_\nu v - \Bi \nu \cdot (D+A)w=0\ &\ \mbox{on}\ \Gamma_h^{\pm}, 
      \end{cases}
  \end{equation}
where $\nu \in \mathbb{S}^1$ signifies the exterior unit normal vector to $\Gamma_h^{ \pm}$ and { $ k \in \mathbb{C} \backslash\{0\}$}. If { $\nabla w \in {L^{{\frac{2}{1-\varepsilon}}}(S_{h})}$}, where $0<\varepsilon <1$, and $\mathbf{A}\in H^{2}(S_{h})$ satisfying 
    \begin{equation*}
      \BA(\Bx_c) \cdot \mathbf{H}(\theta) \neq 0, 
    \end{equation*}
    where $\mathbf{H}(\theta) = (\sin (\theta_M - \theta_m) ,  \Bi \sin({\theta_M - \theta_m}) )$, then $v(\mathbf{0})=w(\mathbf{0})=0$.
\end{thm}

\begin{proof}
  
  Without loss of generality, we assume that $\Bx_{c}=\mathbf{0}$. Consider the following integral:
  \begin{equation}\label{Def:I1}
        \int_{S_h} [(D+\BA)^{2}w + \Delta v]u_{0}(s\Bx)\mathrm{d} \Bx =\int_{S_h} (k^{2}w-qw -k^{2}v)u_{0}(s\Bx)\mathrm{d} \Bx :=I_1.
\end{equation}

On the other hand, by applying Green's formula and utilizing the boundary condition in (\ref{main:TEF}), one can deduce that
\begin{equation}
  \begin{split}
    &\int_{S_h} [(D+\BA)^{2}w + \Delta v]u_{0}(s\Bx)\mathrm{d} \Bx \\
    =&\lim _{\epsilon  \rightarrow 0} \int_{\Omega_{\epsilon}}[(D+\BA)^{2}w + \Delta v]u_{0}(s\Bx) \mathrm{d} \Bx\\
    =& \Bi \int_{\Gamma_{h}^{\pm}}  \nu \cdot \BA v u_{0}(s \Bx)  \mathrm{d} \sigma + \int_{\Lambda_{h}}\left(u_{0}(s \Bx) \partial_{\nu}(v-w)-(v-w) \partial_{\nu} u_{0}(s \Bx)\right) \mathrm{d} \sigma  \\
    & -\Bi \int_{S_h}  (\mathbf{A}\cdot \nabla w + \nabla \cdot (\mathbf{A}w)) u_{0}(s\Bx) \mathrm{d}\Bx + \int_{S_h} \mathbf{A}^2 wu_{0}(s\Bx) \mathrm{d}\Bx.
  \end{split}
\end{equation}
where $\Omega_{\epsilon}=S_{h}\setminus B_{\epsilon }$ for $\epsilon >0$. For the sake of notation uniformity, we define
\begin{equation}\label{I2345}
  \begin{aligned}
    I_{2}^{\pm} &= \int_{\Gamma_{h}^{\pm}} \nu \cdot \BA v \, u_{0}(s \Bx) \, \mathrm{d} \sigma, \\
    I_{3} &= \int_{\Lambda_{h}} ( u_{0}(s \Bx) \partial_{\nu}(v-w) - (v-w) \partial_{\nu} u_{0}(s \Bx) ) \mathrm{d} \sigma, \\
    I_{4} &= \int_{S_h} \left( \mathbf{A} \cdot \nabla w + \nabla \cdot (\mathbf{A} w) \right) u_{0}(s \Bx) \, \mathrm{d} \Bx, \\
    I_{5} &= \int_{S_h} \mathbf{A}^2 w \, u_{0}(s \Bx) \, \mathrm{d} \Bx.
\end{aligned}
\end{equation}

By the Sobolev embedding theorem, we know that $ v \in \mathcal{C}^{\alpha}(\overline{S}_{h})$  and $\BA(\Bx) \in \mathcal{C}^{\alpha}(\overline{S}_{h})$ for $0<\alpha<1$. Then by Lemma $\ref{Lemma:expansion}$, we can obtain
\begin{equation}\label{I2}
  I_{2}^{\pm} = \nu^{\pm} \cdot \BA(\mathbf{0}) v(\mathbf{0}) \int_{\Gamma_{h}^{\pm}} u_{0}(s \Bx) \mathrm{d} \sigma + \int_{\Gamma_{h}^{\pm}} \delta \nu \cdot \mathbf{A}v u_{0}(s \Bx) \mathrm{d} \sigma,
\end{equation}
where $\nu^{\pm}$ denote the unit external normal vector on $\Gamma_{h}^{\pm}$ respectively. Furthermore, we can derive that
\begin{equation*}
  \int_{\Gamma_{h}^{+}} u_{0}(s \Bx) \mathrm{d} \sigma = \int_{0}^{\infty} u_{0}(s \Bx)|_{\Gamma_{h}^{+}} \mathrm{d} \Bx - \int_{h}^{\infty} u_{0}(s \Bx)|_{\Gamma_{h}^{+}} \mathrm{d} \Bx,
\end{equation*}
with
\begin{equation*}
    \int_{0}^{\infty} u_{0}(s \Bx)|_{\Gamma_{h}^{+}} \mathrm{d} \Bx = \int_{0}^{\infty} e^{-\sqrt{sr}\mu(\theta_{M})} \mathrm{d} r
    = \int_{0}^{\infty} 2s^{-1}\mu^{-2}(\theta_{M}) e^{-t}t \mathrm{d} t
    =4\mu^{-2}(\theta_{M}) s^{-1},
\end{equation*}
and 
\begin{equation*}
  \int_{h}^{\infty} | u_{0}(s \Bx)|_{\Gamma_{h}^{+}} | \mathrm{d} \Bx = \int_{h}^{\infty} e^{-\sqrt{sr}\omega(\theta_{M})} \mathrm{d} r  \leq 8 s^{-1} {\omega^{-2}(\theta_{M})} e ^{-\sqrt{sh}\omega(\theta_{M})/2 },
\end{equation*}
where 
\begin{equation*}
  \omega(\theta)=-\cos (\theta / 2+\pi), \quad \mu(\theta)=-\cos (\theta / 2+\pi)-\mathbf{i} \sin (\theta / 2+\pi).
\end{equation*}

Similarly, we can get the integral of $u_{0}(s \Bx)$ over $\Gamma_{h}^{-}$.

For the second integral in \eqref{I2}, combining it with Lemma $\ref{Lemma:expansion}$ yields that
\begin{equation*}
    |\int_{\Gamma_{h}^{\pm}} \delta \nu \cdot \mathbf{A}v u_{0}(s \Bx) \mathrm{d} \sigma|  \leq \|\nu \cdot \BA v\|_{\mathcal{C}^{\alpha}(\Gamma_{h}^{\pm})} \int_{\Gamma_{h}^{\pm}} |u_{0}(s \Bx)| |\Bx|^{\alpha} \mathrm{d} \Bx ,
\end{equation*}
where
\begin{equation*}
  \begin{split}
    \int_{\Gamma_{h}^{+}} |u_{0}(s \Bx)| |\Bx|^{\alpha} \mathrm{d} \Bx & = \int_{0}^{h} e^{-\sqrt{sr}\omega(\theta_{M})} r^{\alpha} \mathrm{d} r \\
    &= \int_{0}^{\sqrt{sh}\omega(\theta_{M})} 2 s^{-1-\alpha} {\omega(\theta_{M})}^{-2-2\alpha} e^{-t} t^{2\alpha +1} \mathrm{d} t\\
   &\leq 2\Gamma(2\alpha +2) \omega^{-2-2\alpha}(\theta_{M}) s^{-1-\alpha}.
  \end{split}
\end{equation*}
Therefore, we finally get that
\begin{equation}
  \begin{split}
    I_{2}^{\pm} = & 4 [\frac{\nu^{-}}{\mu^{2}(\theta_{m})} + \frac{\nu^{+}}{\mu^{2}(\theta_{M})}] \cdot \BA(\mathbf{0}) v(\mathbf{0}) s^{-1} \\
      &-\nu^{\pm} \cdot \BA(\BO) v(\mathbf{0}) \int_{h}^{\infty} u_{0}(s \Bx)|_{\Gamma_{h}^{\pm}} \mathrm{d} \Bx + \int_{\Gamma_{h}^{\pm}} \delta \nu \cdot \mathbf{A}v u_{0}(s \Bx) \mathrm{d} \sigma\\
      :=& 4 [\frac{\nu^{-}}{\mu^{2}(\theta_{m})} + \frac{\nu^{+}}{\mu^{2}(\theta_{M})}] \cdot \BA(\mathbf{0}) v(\mathbf{0}) s^{-1} + I_{21}^{\pm} + + I_{22}^{\pm}.
  \end{split}
\end{equation}
Therefore, we can get that
\begin{equation}\label{main:eq1}
   4\Bi [\frac{\nu^{-}}{\mu^{2}(\theta_{m})} + \frac{\nu^{+}}{\mu^{2}(\theta_{M})}] \cdot \BA(\mathbf{0}) v(\mathbf{0}) s^{-1} + \Bi I_{21}^{\pm} + \Bi I_{22}^{\pm}  = I_{1} - I_{3} +\Bi I_{4} - I_{5}.
\end{equation}

For $I_{1}$ definded in \eqref{Def:I1}, using the embedding theorem, H\"older's inequality and Corollary \ref{cor2.2}, one can deduce that
\begin{equation}\label{ES:I1}
  \begin{split}
    |I_{1}| &\leq \|k^{2}w -qw - k^{2}v \|_{L^{{\frac{2}{1-\varepsilon}}}(S_{h})} \|u_{0}(s\Bx)\|_{L^{{\frac{2}{1+\varepsilon}}}(S_{h})} \\
    &\leq C \|k^{2}w -qw - k^{2}v \|_{H^{1}(S_{h})} \|u_{0}(s\Bx)\|_{L^{{\frac{2}{1+\varepsilon}}}(S_{h})} \\
    &\leq C s^{-1-\varepsilon}.
  \end{split}
\end{equation}
Similarly, we can get that
\begin{equation}\label{ES:I45}
  \begin{split}
    |I_{4}| &\leq (\|\BA \cdot \nabla w \|_{{L^{{\frac{2}{1-\varepsilon}}}(S_{h})}} + \|\nabla \cdot (\BA w)\|_{{L^{{\frac{2}{1-\varepsilon}}}(S_{h})}}) \|u_{0}(s\Bx)\|_{L^{{\frac{2}{1+\varepsilon}}}(S_{h})} \leq C s^{-1-\varepsilon},\\
    |I_{5}| &\leq |\BA|^{2}_{L^{\infty} (S_{h})} \|w\|_{{H^{\varepsilon}(S_{h})}} \|u_{0}(s\Bx)\|_{L^{{\frac{2}{1+\varepsilon}}}(S_{h})} \leq C s^{-1-\varepsilon}.
  \end{split}
\end{equation}
Using the H$\ddot{\mbox{o}}$lder inequality and the trace theorem, we further get that 
\begin{equation}\label{ES:I3}
  \begin{split}
      \left|I_{3}\right| & \leq\left\|u_{0}(s \Bx)\right\|_{H^{1/2}\left(\Lambda_{h}\right)}\left\|\partial_{\nu}(v-w)\right\|_{H^{-1/2}\left(\Lambda_{h}\right)}+\left\|\partial_{\nu} u_{0}(s \Bx)\right\|_{L^{2}\left(\Lambda_{h}\right)}\|v-w\|_{L^{2}\left(\Lambda_{h}\right)} \\
      & \leq\left(\left\|u_{0}(s \Bx)\right\|_{H^1\left(\Lambda_{h}\right)}+\left\|\partial_{\nu} u_{0}(s \Bx)\right\|_{L^{2}\left(\Lambda_{h}\right)}\right)\|v-w\|_{H^{1}\left(\Sigma_{\Lambda_{h}}\right)} \leq C e^{-c^{\prime} \sqrt{s}},
  \end{split}
\end{equation}
where $c^{\prime}>0$ as $s \rightarrow \infty$.

Therefore, multiply both sides by $s$ in \eqref{main:eq1} and let $s \to + \infty$ , we can get that 
\begin{equation}
  4 [\frac{\nu^{-}}{\mu^{2}(\theta_{m})} + \frac{\nu^{+}}{\mu^{2}(\theta_{M})}] \cdot   \BA(\mathbf{0}) v(\mathbf{0})  = 0,
\end{equation}
Using the fact that $\nu^{-} = (-\sin{\theta_m},\cos{\theta_m})$ and $\nu^{+} = (\sin{\theta_M},-\cos{\theta_M})$, we know that 
\begin{equation}\label{H}
  \frac{\nu^{-}}{\mu^{2}(\theta_{m})} + \frac{\nu^{+}}{\mu^{2}(\theta_{M})} = \frac{ \mathbf{H}(\theta)}{\left(\cos \theta_{m}+\mathbf{i} \sin \theta_{m}\right)\left(\cos \theta_{M}+\mathbf{i} \sin \theta_{M}\right)}
\end{equation}
where $\mathbf{H}(\theta) = (\sin (\theta_M - \theta_m) ,  \Bi \sin({\theta_M - \theta_m}) )$. Combine it with the condition that $\BA(\BO) \cdot \mathbf{H}(\theta) \neq 0$, we can get that $v(\BO) = 0$.

The proof is complete.

\end{proof}

\subsection{Vanishing properties in 3D}
Similar to Theorem \ref{thm:2D}, the eigenfunctions display a vanishing characteristic at three-dimensional corner points. We now introduce the concept of edge corner geometry in a three-dimensional context. This notation is also discussed in \cite{DCL}. Let \( W \) be a sector as defined in (\ref{W}), and let \( M \in \mathbb{R}_{+} \). It is evident that \( W \times (-M, M) \) describes an edge singularity, which we refer to as a 3D corner. In the following discussion, consider a Lipschitz domain \(\Omega \subset \mathbb{R}^{3}\) that contains a 3D corner. Let \(\mathbf{x}_{c} \in \mathbb{R}^{2}\) be the vertex of \( W \), and \( x_{3}^{c} \in (-M, M) \). Thus, \((\mathbf{x}_{c}, x_{3}^{c})\) is defined as an edge point of \( W \times (-M, M) \). We will use \(\mathbf{x} = (\mathbf{x}', x_3) \in W \times (-M, M)\) to denote a point in the edge corner. See Figure \ref{fg:02} for a 3D illustration of the geometry.

\begin{figure}[htbp]
  \centering
  \includegraphics[scale=0.3]{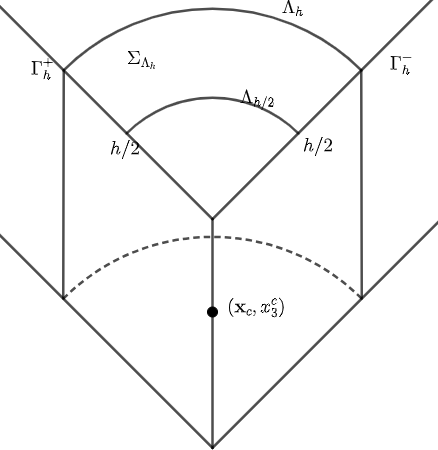}
  \caption{Illustration of the geometry in 3D} \label{fg:02}
\end{figure}

By introducing a dimensionality reduction operator, we can utilize results from the two-dimensional case to simplify the proof of the vanishing property at corner points in the three-dimensional context.
\begin{definition}\label{Definition 3.1}
  \cite[Definition 3.1]{DCL} Let $W \subset \mathbb{R}^{2}$ be defined in \eqref{W}, $M>0 .$ For a given function
$g$ with the domain $W \times(-M, M) .$ Pick up any point $x_{3}^{c} \in(-M, M) .$ Suppose $\psi \in C_{0}^{\infty}\left(\left(x_{3}^{c}-L, x_{3}^{c}+L\right)\right)$ is a nonnegative function and $\psi \neq 0,$ where $L$ is sufficiently small such that $\left(x_{3}^{c}-L, x_{3}^{c}+L\right) \subset(-M, M),$ and write $\Bx=\left(\Bx^{\prime},  x_{3}\right) \in \mathbb{R}^{3},  \Bx^{\prime} \in \mathbb{R}^{2}$.
The dimension reduction operator $\CR$ is defined by
\begin{equation}
  \CR(g)\left(\Bx^{\prime}\right)=\int_{x_{3}^{c}-L}^{x_{3}^{c}+L} \psi\left(x_{3}\right) g\left(\Bx^{\prime}, x_{3}\right) \mathrm{d} x_{3}.
\end{equation}
where $\Bx^{\prime} \in W$.
\end{definition}
For later usage, we need the following regularity result of the dimension reduction operator.
\begin{Lemma}\label{Lemma 3.1}
  \cite[Lemma 3.1]{DDL} Let $g \in H^{2}(W \times(-M, M)) \cap C^{\alpha}(\bar{W} \times[-M, M]),$ where $0<\alpha<1$. Then
\begin{equation*}
  \CR(g)\left(\Bx^{\prime}\right) \in H^{2}(W) \cap C^{\alpha}(\bar{W}).
\end{equation*}
\end{Lemma}

Similar to Theorem \ref{thm:2D}, the following theorem describes the vanishing properties of eigenfunctions in three dimensions.
\begin{thm}\label{thm:3D}
  Let $S_{h}$ be defined in \eqref{Sh}. Suppose that $v\in H^2(S_h\times (-M,M))$ and $w\in H^1(S_h\times (-M,M)) $ satisfy the following PDE system:
  \begin{equation}\label{main:TEF3}
    \begin{cases}
      L_{\BA,q} w -k^{2}w=0\ &\ \mbox{in}\ \ S_{h}\times (-M,M),\medskip\\
    (\Delta+k^2) v=0\ &\ \mbox{in}\ \ S_{h}\times (-M, M),\medskip\\
    w=v,\ \ \Bi \nu \cdot (D+\BA)w=\partial_\nu v\ &\ \mbox{on}\ \Gamma^\pm\times (-M, M),
    \end{cases}
  \end{equation}
where $\nu \in \mathbb{S}^1$ signifies the exterior unit normal vector to $\Gamma_h^{ \pm}$ and $ k \in \mathbb{C} \backslash\{0\}$. If { $\nabla \CR(w) \in {L^{{\frac{2}{1-\varepsilon}}}(S_{h})}$}, where $\varepsilon >0$, and $\BA = \BA(\Bx^{\prime})\in H^{2}(S_{h})$ is independent of $x_{3}$ satisfying
    \begin{equation*}
      \BA((\Bx_c,x_3^c)) \cdot \mathbf{G}(\theta) \neq 0, 
    \end{equation*}
    where $\mathbf{G}(\theta) = (\sin (\theta_M - \theta_m) ,  \Bi \sin({\theta_M - \theta_m}),0 )$, then $v(\mathbf{0})=w(\mathbf{0})=0$.
\end{thm}

\begin{proof}
  For an edge point $\left(\Bx_{c}, x_{3}^{c}\right) \in W \times(-M, M),$ we assume, with out loss of generality, that the vertex $\left(\Bx_{c}, x_{3}^{c}\right)=\mathbf{0}$. By direct calculations, we have
  \begin{equation}
    \begin{split}
       \Delta_{\Bx^{\prime}} \CR(v)\left(\Bx^{\prime}\right)&=\Delta_{\Bx^{\prime}} \int_{-L}^{L} \psi(x_3)v(\Bx^{\prime},x_3) \mathrm{d}x_3\\
       &=\int_{-L}^{L} \psi(x_3) \left(-k^{2}v(\Bx^{\prime},x_3)-\partial_{x_3}^{2}v(\Bx^{\prime},x_3)\right) \mathrm{d}x_3\\
       &=-\int_{-L}^{L} \psi(x_3)\partial_{x_3}^{2}v(\Bx^{\prime},x_3)\mathrm{d}x_3 - k^{2}\CR(v)(\Bx^{\prime})\\
       &=-\int_{-L}^{L} \psi^{\prime \prime}\left(x_3\right) v\left(\Bx^{\prime}, x_3\right) \mathrm{d}x_3-k^{2} \CR(v)\left(\Bx^{\prime}\right).
    \end{split}
\end{equation}
Similarly, we can obtain that
\begin{equation}
 \begin{split}
   &(D_{\Bx'}+\BA^{\prime})^{2} \CR(w) (\Bx')\\ 
   &= (D_{\Bx'}+\BA^{\prime})^{2} \int_{-L}^{L} \psi(x_3)w(\Bx^{\prime},x_3) \mathrm{d}x_3\\
   &= \int_{-L}^{L} \psi(x_3) ((k^2- q - A_{3}^2)w + \partial_{x_{3}}^{2}w + 2\Bi A_{3} \partial_{x_{3}} w )\mathrm{d}x_3\\
   &= \int_{-L}^{L} \psi^{\prime \prime}\left(x_3\right) w\left(\Bx^{\prime}, x_3\right) \mathrm{d}x_3 + k^{2} \CR(q^{'}w)\left(\Bx^{\prime}\right) - 2\Bi A_{3} \int_{-L}^{L} \psi^{\prime }\left(x_3\right) w\left(\Bx^{\prime}, x_3\right) \mathrm{d}x_3,
 \end{split}
\end{equation}
where $\BA^{\prime} = (A_1,A_2)$ and $q^{'} = 1-\frac{q}{k^2} - \frac{A_{3}^{2}}{k^2}$. 
Therefore, we have
\begin{equation}\label{F1234}
  \begin{split}
    & (D_{\Bx'}+\BA^{\prime})^{2} \CR(w) (\Bx') + \Delta_{\Bx^{\prime}} \CR(v)\left(\Bx^{\prime}\right)\\
    =& \int_{-L}^{L} \psi^{\prime \prime}\left(x_3\right)\left(w\left(\Bx^{\prime}, x_3\right)-v\left(\Bx^{\prime}, x_3\right)\right) \mathrm{d} x_3+k^{2} \CR(q^{'} w)\left(\Bx^{\prime}\right)-k^{2} \CR(v)\left(\Bx^{\prime}\right)\\
    & - 2\Bi A_{3} \int_{-L}^{L} \psi^{\prime }\left(x_3\right) w\left(\Bx^{\prime}, x_3\right) \mathrm{d}x_3\\
    :=& F_{1}\left(\Bx^{\prime}\right)+ F_{2}\left(\Bx^{\prime}\right)+ F_{3}\left(\Bx^{\prime}\right) + \Bi F_{4}\left(\Bx^{\prime}\right).
  \end{split}
\end{equation}
Consider the following integral
\begin{equation}
  \begin{split}
    &\int_{S_{h}}\left((D_{\Bx'}+\BA^{\prime})^{2} \CR(w) (\Bx') + \Delta_{\Bx^{\prime}} \CR(v)\left(\Bx^{\prime}\right)\right)u_{0}\left(s \Bx^{\prime}\right)\mathrm{d} \Bx^{\prime} \\
    =& \int_{S_{h}}\left(F_{1}\left(\Bx^{\prime}\right)+F_{2}\left(\Bx^{\prime}\right)+F_{3}\left(\Bx^{\prime}\right) + F_{4}\left(\Bx^{\prime}\right)\right) u_{0}\left(s \Bx^{\prime}\right) \mathrm{d} \Bx^{\prime}\\
    := & I_{1}
  \end{split}
\end{equation}

On the other hand, since $w\left(\Bx^{\prime}, x_{3}\right)=v\left(\Bx^{\prime}, x_{3}\right)$ when $\Bx^{\prime} \in \Gamma$ and $-L<x_{3}<L$, we have
\begin{equation}
    \CR(w)\left(\Bx^{\prime}\right)=\CR(v)\left(\Bx^{\prime}\right) \text { on } \Gamma^{\pm}.
\end{equation}
Similarly, using the fact that $\BA$ is independent of $x_3$ and $\nu=(-\sin{\theta_{m}},\cos{\theta_{m}},0)$ on $\Gamma^{-} \times (-M,M)$, we can obtain that
\begin{equation}\label{BC}
   \Bi \nu \cdot (D_{\Bx'}+\BA^{'}) \CR(w)\left(\Bx^{\prime}\right) = \partial_{\nu} \CR(v)\left(\Bx^{\prime}\right) \text { on } \Gamma^{\pm}.
\end{equation}
Therefore, by Green's formula, we have
\begin{equation}
  \begin{split}
    &\int_{S_h} [(D_{\Bx'}+\mathbf{A}^{\prime})^2 \CR (w)\left(\Bx^{\prime}\right) +\Delta_{\Bx '} \CR(v)\left(\Bx^{\prime}\right)] u_{0}(s\Bx') \mathrm{d} \Bx' \\
    =&\lim _{\varepsilon \rightarrow 0} \int_{D_{\varepsilon}}[(D_{\Bx'}+\mathbf{A}^{\prime})^2 \CR (w)\left(\Bx^{\prime}\right) +\Delta_{\Bx '} \CR(v)\left(\Bx^{\prime}\right)] u_{0}(s\Bx') \mathrm{d} \Bx'\\
    =&\int_{\Gamma_{h}^{\pm}}   \Bi \nu \cdot \BA' \CR(v) u_{0}(s \Bx') \mathrm{d} \sigma + \int_{\Lambda_{h}}\left(u_{0}(s \Bx') \partial_{\nu}\CR(v-w)-\CR(v-w) \partial_{\nu} u_{0}(s \Bx')\right) \mathrm{d} \sigma  \\
    & - \int_{S_h} \mathbf{i} (\mathbf{A}^{'}\cdot \nabla_{\Bx '} \CR(w) + \nabla_{\Bx '} \cdot (\mathbf{A}^{'}\CR(w))) u_{0}(s\Bx') \mathrm{d}\Bx' + \int_{S_h} \mathbf{A}^{'2} \CR(w)u_{0}(s\Bx') \mathrm{d}\Bx'\\
    :=& \Bi I_{2}^{\pm} + I_{3} -\Bi I_{4} + I_{5}.
  \end{split}
\end{equation}
Therefore, we can deduce that
\begin{equation}\label{main:eq2}
  I_1 = I_3 + \Bi I_{2}^{\pm} - \Bi I_4 + I_5.
\end{equation}

Similar to the analysis of \eqref{I2}, we can get that 
\begin{equation}
  I_{2}^{\pm}=4 [\frac{\nu^{-}}{\mu^{2}(\theta_{m})} + \frac{\nu^{+}}{\mu^{2}(\theta_{M})}] \cdot \BA^{\prime}(\mathbf{0}) \CR(v)(\mathbf{0}) s^{-1} + I_{21}^{\pm} + I_{22}^{\pm},
\end{equation}
where
\begin{equation*}
  I_{21}^{\pm} \leq \mathcal{O} (s^{-1} e ^{-\sqrt{sh}\omega(\theta_{M})/2 }) , \quad I_{22}^{\pm} \leq \mathcal{O} (s^{-1-\alpha}).
\end{equation*}
By lemma \ref{Lemma 3.1}, and the same arguments in \eqref{ES:I1}, \eqref{ES:I45} and \eqref{ES:I3}, we have 
\begin{equation}\label{I1345}
    |I_{1}|  \leq C_1 s^{-1-\varepsilon} ,\ 
    \left|I_{3}\right|  \leq C_3 e^{-c^{\prime} \sqrt{s}},\  
    |I_{4}|  \leq C_4 s^{-1-\varepsilon},\  
    |I_{5}|  \leq C_5 s^{-1-\varepsilon} ,
\end{equation}
where $C_j \leq 0$ for $j=1,3,4,5$ and $c'>0$ as $s \to \infty$.

Therefore, multiply both sides by $s$ in \eqref{main:eq2} and let $s \to + \infty$ , we can get that 
\begin{equation}
  4 [\frac{\nu^{-}}{\mu^{2}(\theta_{m})} + \frac{\nu^{+}}{\mu^{2}(\theta_{M})}] \cdot  \BA^{\prime}(\mathbf{0}) \CR(v)(\mathbf{0})  = 0,
\end{equation}
Combine it with condition $A^{\prime}(\BO) \cdot \mathbf{H}(\theta) \neq 0$, one can get $\CR(v)(\mathbf{0}) = 0$ which implies that $v(\mathbf{0}) = 0$.

\end{proof}

\section{Unique recovery results}
In this section, we utilize the vanishing properties of eigenfunctions, as established in Theorem \ref{thm:2D} and Theorem \ref{thm:3D}, to demonstrate the unique recovery results for the inverse scattering problem \eqref{ISP}, as detailed in Theorem \ref{main:thm2}. The technical requirement \(\nabla w \in L^{{\frac{2}{1-\varepsilon}}}(S_{h})\) in Theorem \ref{thm:2D} is readily satisfied. Indeed, this condition follows from classical results on the singular behavior of solutions to elliptic partial differential equations (PDEs) in corner domains \cite{CFTSN, CMMD, DM}. Specifically, it is known that the solution can be decomposed into a singular part and a regular part, with the singular part exhibiting Hölder continuity that depends on the corner geometry, the boundary, and the right-hand side inputs. For our subsequent analysis, we first present the following result in a relatively simple scenario.

In \eqref{main:eqs}, standard PDE theory (see, e.g., \cite{MW}) indicates that the solution \( u \) is real-analytic away from the boundary. Denote
$$
S_{2 h}=W \cap B_{2 h}, \quad \Gamma_{2 h}^{ \pm}=\Gamma^{ \pm} \cap B_{2 h},
$$
where $W$ is the sector defined in \eqref{Sh}, $B_{2 h}$ is an open ball centered at $\mathbf{0}$ in $\mathbb{R}^2$ with the radius $2 h$ and $\Gamma^{ \pm}$are the boundaries of $W$.

\begin{lemma}\label{Reg}
  Suppose that $u \in H^{1}(B_{2h})$ satisfies
  \begin{equation}\label{leq}
    \begin{cases}
      L_{\BA,q} u^{-} -k^{2}u^{-}=0\ &\ \mbox{in}\ \ S_{2h},\medskip\\
      (\Delta+k^2) u^{+}=0\ &\ \mbox{in}\ \ B_{2h}\setminus \overline{S_{2h}},\medskip\\
      u^+=u^-,\ \ \Bi \nu \cdot (D+A)u^-=\partial_\nu u^+\ &\ \mbox{on}\ \Gamma^{\pm}_{2h},
      \end{cases}
  \end{equation}
  where $u^{+} = u|_{B_{2h}\setminus \overline{S_{2h}}}$, $u^{-} = u|_{S_{2h}}$ and $k$ is complex constant. Then there exist $0< \varepsilon <1$ such that  $ \nabla u^{-} \in {L^{{\frac{2}{1-\varepsilon}}}(S_{h})}$.
\end{lemma}

\begin{proof}
  We first introduce a triangle $V_{2h}$ in the following: Select two points $\mathbf{y}_1 \in \Gamma^{-}_{2h} \setminus \Gamma^{-}_{h}$ and $\mathbf{y}_2 \in \Gamma^{+}_{2h} \setminus \Gamma^{+}_{h}$, and connect them. The triangle  $V_{2h}$ is then defined by the vertices $\mathbf{y}_1, \mathbf{y}_2$ and $\Bx_{c}$. The line segment $\Lambda^{'}_{2h}$ represents the line connecting $\mathbf{y}_1$ and $\mathbf{y}_2$. See Figure \ref{fg:04} for an illustration.
  \begin{figure}[htbp]
    \centering
    \includegraphics[scale=0.3]{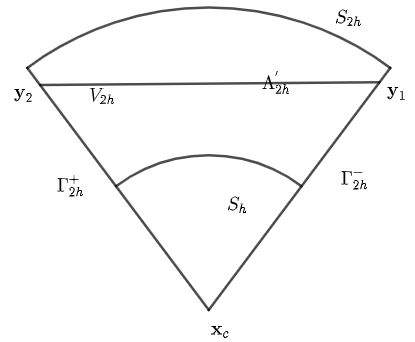}
    \caption{Illustration of $V_{2h}$} \label{fg:04}
  \end{figure}
  
  Since $u^{+}$ is real analytic in $B_{2h}\setminus \overline{S_{2h}}$, we define  $\bar{u}^{+}$ as the analytic extension of $u^{+}$ within $B_{2h}$. Let $w= u^{-} - \bar{u}^{+}$. Then, from \eqref{leq}, it follows that 
  \begin{equation*}
    \begin{cases}
      \Delta w=f\ &\ \mbox{in}\ \ V_{2h},\medskip\\
      w=\psi\ &\ \mbox{in}\ \ \Lambda^{'}_{2h},\medskip\\
      w=0,\ &\ \mbox{on}\ \Gamma^{\pm}_{2h},
      \end{cases}
  \end{equation*}
where $f=(k^{2}q' + \Bi \nabla \cdot \BA - \BA^{2})\bar{u}^{+} + 2\Bi \BA \cdot \nabla \bar{u}^{+} - k^{2} u^{-} \in L^{2}(S_{2h})$ and $\psi = w|_{\Lambda^{'}_{2h}} \in C^{\infty}\left(\bar{\Lambda}_h\right)$.  Then by \cite[Theorem 2.1]{MZ}, there exists the following decomposition
$$
u=\tilde{u}_R+{c} r^{\lambda} \sin (\lambda \theta), 
$$
where $\tilde{u}_R \in H^2\left(S_h\right)$ denotes the regular part, and  ${c} r^{\lambda} \sin (\lambda \theta)$ represents the singularity of the solution, with  $\lambda = \frac{\pi}{\theta_M - \theta_m}$. For $0<\varepsilon <1$, we have $\lambda > \frac{1+\varepsilon }{2}$, which implies that ${c}(f)  r^{\lambda} \sin (\lambda \theta) \in W^{1}_{\frac{2}{1-\varepsilon}}$ by \cite[Theorem 2.1]{MZ}. Therefore, by the Sobolev embedding theorem, it follows that $ \nabla u^{-} \in {L^{{\frac{2}{1-\varepsilon}}}(S_{h})}$.

\end{proof}

Recovering an object's shape from a single measurement in inverse scattering problems is challenging due to the problem's inherent nonlinearity, incomplete and potentially noisy data, instability of inverse problems, the possibility of multiple solutions, and computational complexity. The existing literature often employs various strategies, such as using measurement data from multiple angles or frequencies, introducing regularization techniques to stabilize the solution, utilizing prior information to reduce solution multiplicity, and developing more efficient numerical algorithms to manage computational demands (see \cite{LHJZ,FDH,JCMY,OIA}). In this paper, we require the following prior information about the magnetic and electric potentials.

\begin{definition}\label{AS}
  (Admissible shape). Let $(\Omega;k,d,q,\BA)$ be a scatterer associated with the incident wave $u^{i} = e^{ikx\cdot d}$, where $k\in \mathbb{R}_{+}$.  Consider the scattering problem described by \eqref{main:eqs}, with $u$ representing the wave function. The scatterer is deemed admissible if it satisfies the following conditions:
  \begin{itemize}
    \item $\Omega$ is a bounded simply connected Lipschitz domain in $\mathbb{R}^n$ and electric potential $q\in L^{\infty}(\Omega, \mathbb{C})$ , magnetic potential $\BA =(A_j)_{j=1}^{n} \in L^{\infty}\left(\Omega, \mathbb{C}^n\right)$.
    
    \item  In 2D, following the notation in \eqref{Sh},  if $\Omega$ possesses a corner $B_h\left(\Bx_c\right) \cap \Omega=\Omega \cap W_{\Bx_c}\left(\theta_W\right)$ where $\Bx_c$ is the vertex of the sector $W_{\Bx_c}\left(\theta_W\right)$, then $\mathbf{A}\in H^{2}(S_{h})$ and 
    \begin{equation*}
      \BA(\Bx_c) \cdot \mathbf{H}(\theta) \neq 0, 
    \end{equation*}
    where $\mathbf{H}(\theta) = (\sin (\theta_M - \theta_m) ,  \Bi \sin({\theta_M - \theta_m}) )$.

    \item In 3D, if $\Omega$ possesses an edge corner $(B_h(\Bx_c)\times (-M,M)) \cap \Omega=\Omega \cap (W_{\Bx_c}(\theta_W)\times (-M,M))$, then $\mathbf{A}  \in H^{2}(S_{h} \times (-M,M)) $ and $\BA = \BA(\Bx^{\prime})$ is independent of $x_{3}$ satisfying
    \begin{equation*}
      \BA((\Bx_c,x_3^c)) \cdot \mathbf{G}(\theta) \neq 0, 
    \end{equation*}
    where $x_3^c \in (-M,M)$ and $\mathbf{G}(\theta) = (\sin (\theta_M - \theta_m) ,  \Bi \sin({\theta_M - \theta_m}),0 )$.
  \end{itemize}
\end{definition}

\begin{definition}\label{non-v}
  (Non-vanishing total wave). We say that the electric potential and magnetic potential produce a non-vanishing total wave  if, for any incident wave $u^i$, the total wave $u$ does not vanish anywhere.
\end{definition}

\begin{remark}
  We emphasize that, in certain general and practical scenarios, the scatterer \((\Omega; k, d, q, \mathbf{A})\) and the scattering problems described by \eqref{main:eqs} can produce a non-vanishing total wave. Specifically, this condition is satisfied when the magnetic potential \(\mathbf{A}\) and electric potential \(q\) are sufficiently small. According to Theorem \ref{SolE}, the scattering wave is influenced by the electric potential, magnetic potential, and the incident wave in the \(H^1\) norm. Therefore, if the incident wave \(u^i\) is non-vanishing everywhere, such as \(u^i = e^{ik \mathbf{x} \cdot d}\) with \(d \in \mathbb{S}^1\) representing a plane wave, and both \(\mathbf{A}\) and \(q\) are sufficiently small, the total wave will also be non-vanishing. Nonetheless, according to Definition \ref{non-v}, we may consider more general situations in our subsequent analysis of the inverse problem \eqref{ISP}.

\end{remark}

In the following, we demonstrate that, in a general and practical scenario, the polyhedral shape of the scatterer, denoted by \(\Omega\), can be uniquely determined from a single far-field measurement.

\begin{thm}\label{main:thm4}
  Consider the scattering problem described by equation \eqref{main:eqs} associated with two scatterers $\left(\Omega_j ; k, d, q_j, \BA_j\right)$ for $ j=1,2$, in $\mathbb{R}^n$ with $n=2,3$. Let $u_{\infty}^j\left(\hat{\Bx} ; u^i\right)$ denote the far-field pattern corresponding to the scatterer $\left(\Omega_j ; k, d, q_j, \BA_j\right)$ and the incident wave $u^i$. Suppose that the scatterers $\left(\Omega_j ; k, d, q_j, \BA_j\right)$ for $ j=1,2$ are admissible (see Definition \ref{AS}) and produce non-vanishing total waves (see Definition \ref{non-v}), and
    \begin{equation}
      u_{\infty}^1\left(\hat{\Bx} ; u^i\right)=u_{\infty}^2\left(\hat{\Bx} ; u^i\right)
    \end{equation}
  for all $\hat{\Bx} \in \mathbb{S}^{n-1}$ with $n=2,3$ and a fixed incident wave $u^i$. Then
  \begin{equation}
    \Omega_1 \Delta \Omega_2:=\left(\Omega_1 \backslash \Omega_2\right) \cup\left(\Omega_2 \backslash \Omega_1\right)
  \end{equation}
  cannot possess a corner in two dimensions or an edge corner in three dimensions. Hence, if $\Omega_1$ and $\Omega_2$ are convex polygons in $\mathbb{R}^2$ or rectangular boxs in $\mathbb{R}^3$, one must have
  \begin{equation}
    \Omega_1=\Omega_2 .
  \end{equation}
\end{thm}

\begin{proof}
  We first provide a proof of the unique recovery result in two dimensions. Assuming for contradiction that there is a corner in \(\Omega_1 \Delta \Omega_2\), we can, without loss of generality, assert that the vertex \(O\) of the corner \(\Omega_2 \cap W\) is such that \(O \in \partial \Omega_2\) and \(O \notin \bar{\Omega}_1\). Furthermore, we assume without loss of generality that \(O\) is the origin of \(\mathbb{R}^2\), as illustrated in figure \ref{fg:03}.

  \begin{figure}[htbp]
    \centering
    \includegraphics[scale=0.2]{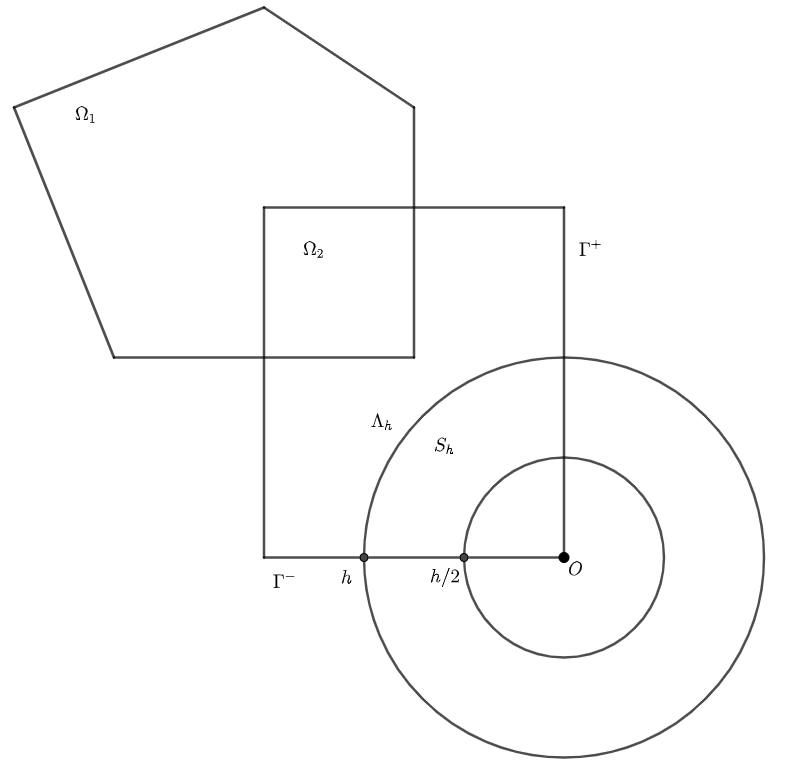}
    \caption{Schematic illustration} \label{fg:03}
  \end{figure}

  Since $u_{\infty}^1\left(\hat{x} ; u^i\right)=u_{\infty}^2\left(\hat{x} ; u^i\right)$ for all $\hat{x} \in \mathbb{S}^1$, we can apply Rellich's Theorem (see \cite[Lemma 2.12]{DCRK}), which implies that $u_1^s=u_2^s$ in $\mathbb{R}^2 \backslash\left(\bar{\Omega}_1 \cup \bar{\Omega}_2\right)$. Therefore,
  \begin{equation}\label{Uniqueness}
    u_1(x)=u_2(x)
  \end{equation}
for all $x \in \mathbb{R}^2 \backslash\left(\bar{\Omega}_1 \cup \bar{\Omega}_2\right)$. Using the notations from \eqref{Sh}, we can infer from \eqref{Uniqueness} that
  \begin{equation*}
    u_2^{-}=u_2^{+}=u_1^{+}, \quad \partial_{\nu} u_2^{-}=\partial_{\nu} u_2^{+}-\Bi(\nu \cdot \BA_2) u_2^{+}=\partial_{\nu} u_1^{+}-\Bi(\nu \cdot \BA_2) u_1^{+} \text { on } \Gamma_h^{ \pm},
  \end{equation*}
where the superscripts $(\cdot)^{-},(\cdot)^{+}$stand for the limits taken from $\Omega_2$ and $\mathbb{R}^2 \backslash \overline{\Omega_2}$, respectively.  Additionally, assume the neighborhood $S_{h}$ is small enough so that
  $$
  \Delta u_1^{+}+k^2 u_1^{+}=0, \quad L_{\BA_{2},q_{2}} u_2^{-}-k^2  u_2^{-}=0 \text { in } S_{ h} .
  $$
  
Applying Theorem \ref{thm:2D}, we have
  \begin{equation*}
    u_1(0) = 0,
  \end{equation*}
which contradicts to the non-vanishing total wave.
The proof of the unique recovery result in three dimensions is consistent with the above proof, except that Theorem \ref{thm:3D} is used instead of Theorem \ref{thm:2D}.
  
The proof is complete.

\end{proof}

  \newpage

  \renewcommand\refname{References}

\end{document}